\documentclass{article}
\usepackage{amsmath,amssymb,amsfonts}
\usepackage{theorem}
\usepackage{color}
\usepackage{hyperref}

\theoremstyle{change}
\newtheorem{definition}{Definition:}[section]
\newtheorem{proposition}[definition]{Proposition:}
\newtheorem{theorem}[definition]{Theorem:}
\newtheorem{lemma}[definition]{Lemma:}
\newtheorem{corollary}[definition]{Corollary:}
{\theorembodyfont{\rmfamily}
\newtheorem{remark}[definition]{Remark:}
}
{\theorembodyfont{\rmfamily}

}
\newenvironment{proof}
  {{\bf Proof:}}
  {\qquad \hspace*{\fill} $\Box$}
\newcommand{\fa}{\mathfrak{a}}
\newcommand{\fg}{\mathfrak{g}}

\newcommand{\fn}{\mathfrak{n}}
\newcommand{\fp}{\mathfrak{p}}

\newcommand{\Ad}{\operatorname{Ad}}

\newcommand{\inner}{\operatorname{int}}
\newcommand{\cl}{\operatorname{cl}}
\newcommand{\fix}{\operatorname{fix}}

\newcommand{\rmd}{\mathrm{d}}
\newcommand{\rme}{\mathrm{e}}

\newcommand{\EC}{\mathcal{E}}

\newcommand{\MC}{\mathcal{M}}
\newcommand{\OC}{\mathcal{O}}

\newcommand{\SC}{\mathcal{S}}
\newcommand{\UC}{\mathcal{U}}

\newcommand{\WC}{\mathcal{W}}

\newcommand{\AC}{\mathcal{A}}
\newcommand{\FC}{\mathcal{F}}

\newcommand{\F}{\mathbb{F}}
\newcommand{\R}{\mathbb{R}}

\begin{document}

\title{Control systems on flag manifolds and their chain control sets}
\author{V. Ayala\thanks{%
Supported by Conicyt, Proyecto Fondecyt n$%
{{}^\circ}%
1150292$} \\
Instituto de Alta Investigaci\'{o}n\\
Universidad de Tarapac\'{a}\\
Casilla 7D, Arica, Chile \and A. Da Silva\thanks{%
Supported by Fapesp grant $n^{o}$ 2013/19756-8} and L. A. B. San Martin \\
Instituto de Matem\'{a}tica,\thanks{%
The third author was supported by CNPq grant 303755/09-1, FAPESP grant
2012/18780-0, and CNPq/Universal grant 476024/2012-9.}\\
Universidade Estadual de Campinas\\
Cx. Postal 6065, 13.081-970 Campinas-SP, Brasil}
\maketitle

\begin{abstract}
In this paper we shown that the chain control sets of induced systems on flag manifolds coincides with their analogous one defined via semigroup actions. Consequently, any chain control set of the system contains a control set with nonempty interior and, if the number of the control sets with nonempty interior coincides with the number of the chain control sets, then the closure of any control set with nonempty interior is a chain control set.
\end{abstract}

\textbf{\noindent{}AMS 2000 Subject Classification}\textit{. 37B35, 37B25, 93D30}


\section{Introduction}

A right-invariant control system on a connected Lie group $G$ is the family of differential equations given by
\begin{flalign*}
&&\dot{g}(t)=X(g(t))+\sum_{j=1}^mu_i(t)Y^j(g(t)), \; \; \;u\in \UC &&\hspace{-1cm}\left(\Sigma\right)
\end{flalign*}

where $X_0, X_1, \ldots, X_m$ are right invariant vector fields and $\mathcal{U}\subset\mathbb{L}_{\infty }(\mathbb{R},U\subset \mathbb{R}^{m})$ where $U\subset\R^m$ is a compact and convex set with $0\in \inner U$. 

In particular, when $G$ is a noncompact semisimple Lie group one can study induced systems on its flag manifolds. That is, if $\Sigma$ is a right-invariant system on $G$, we have induced affine control systems on every flag manifold $\F_{\Theta}=G/P_{\Theta}$ given by
\begin{flalign*}
&&\dot{x}(t)=f^{\Theta}_{0}\left( x(t)\right)
+\sum \limits_{j=1}^{m}u_{j}(t)f^{\Theta}_{j}\left(x(t)\right) ,\;\;\;  u\in \UC&&\hspace{-1cm}\left(\Sigma_{\Theta}\right).
\end{flalign*}
where $f^{\Theta}_i=(d\pi_{\Theta})X_i$ for $i=0, \ldots, m$ and $\pi_{\Theta}:G\rightarrow \F_{\Theta}$ is the canonical projection.

Studing systems on flag manifolds is desirable for at least two reasons: First, flag manifolds includes, in particular, spheres, projective spaces and Grassmannian spaces. Secondly, right-invariant systems appears in many important applications in engineering and physics such as the orientation of rigid bodies and optimal problems (see \cite{bro} and \cite{Jurd}).

In order to understand the dynamical behavior of $\Sigma_{\Theta}$ one have to analize control and chain control sets of the system. There is a way to analize the chain control sets of $\Sigma_{\Theta}$ through the control flow $\phi^{\Theta}$ induced by the system on the fiber bundle 
\begin{equation*}
\phi^{\Theta} :\mathbb{R}\times \mathcal{U}\times \F_{\Theta}\rightarrow \mathcal{U}\times \F_{\Theta}.\;
\;
\end{equation*}%
In fact, since $\F_{\Theta}$ is a compact manifold, Theorem 4.3.11 of \cite{ck} implies that there exists a bijection between the Morse components of the control flow $\phi^{\Theta}$ and the chain control sets of the system $\Sigma_{\Theta}$. Also, by \cite{cbsm}, all the Morse sets of the flow $\phi^{\Theta}$ are given fiberwise as fixed points in $\F_{\Theta}$ of regular elements of $G$.

On the other hand, since the positive orbit $\SC:=\OC^+(1)$ is a semigroup we have the notion of control and chain control sets associated with $\SC$ (see for instance \cite{bsm}, \cite{sm} and \cite{smt}). In particular, any effective chain control set of $\SC$ contains a control set with nonempty interior.

In this paper we show that both notion actually agree, that is, the chain control sets of the induced systems $\Sigma_{\Theta}$ and the effective chain control sets associated with $\SC$ are actually the same. That implies, in particular, that any given chain control sets of $\Sigma_{\Theta}$ contains a control set with nonempty interior. Moreover, if the number of chain control sets coincides with the number of the control sets with nonempty interior then any chain control set is the closure of the control set that it contains. 

The paper is structured as follows: In Section 2 we give some preliminaries on control and Lie theory and enunciate the some results that will be needed. In Section 3 we state and prove our main result, showing that both notions of chain control sets coincides. As a consequence, we prove that if the number of chain control sets and of control sets with nonempty interior coincide, them any chain control set is the closure of the only control set with nonempty interior that it contains.

\section{Preliminaries}

In this section we introduce all the control theory concepts needed in order to present our main result.

\subsection{Control Theory}

Let $M$ be a Riemannian differentiable manifold and consider the admissible class of control $\mathcal{U\subset }$ $\mathbb{L}_{\infty }(\mathbb{R},U\subset \mathbb{R}^{m}).$ A control affine system on $M$ is the family of differential equations
\begin{equation}
\begin{tabular}{c}
$\  \dot{x}(t)=f_{0}\left( x(t)\right)
+\sum \limits_{j=1}^{m}u_{j}(t)f_{j}\left( x(t)\right) ,\  \ x(t)\in M.$%
\end{tabular}
\label{affinesystem}
\end{equation}
where $f_0, f_1, \ldots, f_m$ are smooth vector fields on $M$. The set $\mathcal{U}$ is a compact metrizable space and the
shift flow 
\begin{equation*}
\theta :\mathbb{R}\times \mathcal{U}\rightarrow \mathcal{U},\; \;
\;(t,u)\mapsto \theta _{t}u:=u(\cdot +t)
\end{equation*}%
is a continuous dynamical system, that is chain transitive (see \cite{ck}).

For $u\in \mathcal{U}$ fixed, we denote by $\varphi (\cdot ,x,u)$ the unique
solution of (\ref{affinesystem}) with $\varphi (0,x,u)=x$. If the
vector fields $f_{0},\ldots ,f_{m}$ are elements of $\mathcal{C}^{\infty }$,
then $\varphi $ has also the same class of differentiability respect to $M$
and the corresponding partial derivatives depend continuously on $(t,x,u)\in 
\mathbb{R}\times M\times \mathcal{U}$ (see Thm. 1.1 of \cite{Ka2}\}).

If we assume that all the solutions determined by $\mathcal{U}$ are defined in the whole time
axis\footnote{This condition is in general restrictive, but for the class of affine control systems that we will study it holds.}, we have the map 
\begin{equation*}
\varphi :\mathbb{R}\times M\times \mathcal{U}\rightarrow M,\; \;
\;(t,x,u)\mapsto \varphi (t,x,u),
\end{equation*}%
called the transition map of the system. The transition map together with
the shift flow determines a skew-product flow 
\begin{equation*}
\phi :\mathbb{R}\times \mathcal{U}\times M\rightarrow \mathcal{U}\times M,\;
\;(t,x,u)\mapsto \phi _{t}(u,x)=(\theta _{t}u,\varphi (t,x,u)),
\end{equation*}%
called the control flow of the system (see Sec. 4.3 of \cite{ck}).

The sets 
\begin{align*}
\mathcal{O}_{\tau }^{+}(x)& :=\left\{ \varphi (\tau ,x,u)\ :\ u\in \mathcal{%
U}\right\} , \\
\mathcal{O}_{\leq \text{ }\tau }^{+}(x)& :=\bigcup_{t\text{ }\in \text{ }%
[0,\tau ]}\mathcal{O}_{t}^{+}(x)\mbox{\quad and\quad }\mathcal{O}%
^{+}(x):=\bigcup_{\tau \text{ }>\text{ }0}\mathcal{O}_{\tau }^{+}(x).
\end{align*}%
are the set of reachable points from $x\in M$ at time $\tau > 0$, the set of reachable points from $x$ up to time $\tau $, and the {\bf positive orbit of $x$}, respectively. Analogously we define the sets $\mathcal{O}_{\tau }^{-}(x), \mathcal{O}_{\leq \text{ }\tau }^{-}(x)$ and $\mathcal{O}^{-}(x)$ to be the set of the points controllable to $x$ in time $\tau>0$, the set of points controllable to $x$ up to time $\tau$ and the {\bf negative orbit of $x$}, respectively.

The affine control system (\ref{affinesystem}) is called locally accessible at $x$ if for all $\tau >0$
the sets $\mathcal{O}_{\leq \text{ }\tau }^{+}(x)$ and $\mathcal{O}_{\leq
\tau }^{-}(x)$ have nonempty interiors. It is called {\bf locally
accessible} if it is locally accessible at every point $x\in M$.  Under the assumption of locally accessibility, it holds that $\inner\OC^+(x)$ is a dense subset of $\OC^+(x)$, for any $x\in M$ (see Lemma 1.2 of \cite{Ka2}).

\begin{definition}
A {\bf control set} of the system (\ref{affinesystem}) is a subset $D\subset M$ satisfying:
\end{definition}

\begin{enumerate}
\item[$(i)$] for each $x\in D$ there is $u\in \mathcal{U}$ with $\varphi (%
\mathbb{R}_{+},x,u)\subset D$

\item[$(ii)$] $D\subset \cl\mathcal{O}^{+}(x)$ for all $x\in D$

\item[$(iii)$] $D$ is maximal with respect to the set inclusion and
properties $(i)$ and $(ii)$.
\end{enumerate}

In \cite{ck}, Proposition 3.2.4, it is shown that a subset $D$ with nonempty interior, that is maximal with the property (ii) of the above definition is also a control set. This fact shows, in particular, that not any control set of (\ref{affinesystem}) has to have nonempty interior.

Let us now fix a metric $\varrho $ on $M$. For $x,y\in M$ and 
$\varepsilon ,\tau >0$, a {\bf controlled $(\varepsilon ,\tau )$-chain} from $x$
to $y$ is given by an integer $n\in \mathbb{N}$ and the existence of three
finite sequences 
\begin{equation*}
x_{0},\ldots ,x_{n}\in M,\text{ }u_{0},\ldots ,u_{n-1}\in \mathcal{U},\text{ 
}t_{0},\ldots ,t_{n-1}\geq \tau \text{ }
\end{equation*}%
such that 
\begin{equation*}
x=x_{0},\text{ }y=x_{n}\text{ and }\varrho (\varphi
(t_{i},x_{i},u_{i}),x_{i+1})<\varepsilon ,\text{ \ for }i=0,\ldots ,n-1.
\end{equation*}

\begin{definition}
A set $E\subset M$ is called a {\bf chain control set} of (\ref{affinesystem}) if it
satisfies the following properties:
\end{definition}

\begin{enumerate}
\item[$i)$] for each $x\in E$ there is $u\in \mathcal{U}$ with $\varphi (%
\mathbb{R},x,u)\subset E$

\item[$ii)$] For all $x,y\in E$ and $\varepsilon ,\tau >0$ there exists an $%
(\varepsilon ,\tau )$-chain from $x$ to $y$ in $M$

\item[$iii)$] $E$ is maximal with respect to the set inclusion and the
properties $(i)$ and $(ii)$ above.
\end{enumerate}

We say that a set $E$ satisfying condition (ii) in the above definition is {\bf controlled chain transitive}.

\begin{remark}
From the general theory, every chain control set of the control system (\ref{affinesystem}) is
closed, but this is not necessarely true for control sets. Moreover, every control set
with nonempty interior is contained in a chain control set if local
accessibility holds (see Sec. 4.3 of \cite{ck}).
\end{remark}

\subsection{Lie Theory and Flag Manifolds}

In order to state and prove our main result, in this section we give some
facts about semisimple Lie groups and their induced flag manifolds.

\subsubsection{Semisimple Lie Groups}

Let $G$ be a connected non-compact semisimple Lie group $G$ with finite
center and Lie algebra $\mathfrak{g}$. Fix a Cartan involution $\zeta :%
\mathfrak{g}\rightarrow \mathfrak{g}$ with associated Cartan decomposition $%
\mathfrak{g}=\mathfrak{k}\oplus \mathfrak{s}$ and let $\mathfrak{a}\subset \mathfrak{s}$ be a maximal Abelian subalgebra and $\fa^+\subset \fa$ a Weyl chamber. Let us denote by $\Pi$, $\Pi^+$ and $\Pi^+:=-\Pi^+$ the {\bf set of roots}, the {\bf set of positive roots} and {\bf set of negative roots}, respectively, associated with $\fa^+$.
The {\bf Iwasawa decomposition} of the Lie algebra $\mathfrak{g}$, associated with the above choices, reads as 
\begin{equation*}
\mathfrak{g}=\mathfrak{k}\oplus \mathfrak{a}\oplus \mathfrak{n}^{\pm }\text{
where }\mathfrak{n}^{\pm }:=\sum_{\alpha \text{ }\in \text{ }\Pi ^{\pm }}%
\mathfrak{g}_{\alpha }.
\end{equation*}
where $\fg_{\alpha}:=\{X\in\fg, \;\;[H, X]=\alpha(H)X, \;\;\mbox{ for all }\;H\in\fa\}.$

Let $K$, $A$ and $N^{\pm }$ be the connected Lie subgroups of $G$ with Lie
algebras $\mathfrak{k}$, $\mathfrak{a}$ and $\mathfrak{n}^{\pm }$,
respectively. The Iwasawa decomposition of the Lie group $G$ is given by $%
G=KAN^{\pm }$. The Weyl group of $\mathfrak{g}$ associated to $\alpha \in
\Pi $ is the finite group generated by the reflections on the hyperplane $%
\ker \alpha .$ Alternatively, the Weyl group can be obtained by the quotient 
$M^{\ast }/M$, where $M^{\ast }$ and $M$ are respectively the normalizer and
the centralizer of $\mathfrak{a}$ in $K$. 

We denote by $\Lambda \subset \Pi ^{+}$ the set of the positive roots which cannot be written as linear combinations of other positive roots. The Weyl groups coincides with the group generated by the reflections associated to $\alpha \in \Lambda $. There is only one involutive element $w_0\in\WC$ such that $w_0\Pi^+=\Pi^-$.

For $\Theta\subset\Lambda$ the parabolic subalgebra of type $\Theta$ is $\mathfrak{p}:=\fn^-(\Theta)\oplus\mathfrak{m}\oplus 
\mathfrak{a}\oplus \mathfrak{n}^{+}$ where $\mathfrak{m}$ is the Lie algebra of $M$ and $\fn^-(\Theta)$ is the sum of the eigenspaces $\fg_{\alpha}$ when $\alpha\in\langle\Theta\rangle\cap\Pi^-$, where $\langle\Theta\rangle\subset\Pi$ is the set of roots given as linear combination of the simple roots in $\Theta$. The corresponding parabolic subgroup $P_{\Theta}$ is the normalizer of $\mathfrak{p}_{\Theta}$ in $G$. The flag manifold $\mathbb{F}_{\Theta}$ is the orbit $\mathbb{F}_{\Theta}:=\Ad(G)\mathfrak{p}_{\Theta}$ on a Grasmannian manifold. It is identified with the homogeneous space $G/P_{\Theta}$. When $\Theta=\emptyset$, the sets $\fp:=\fp_{\emptyset}$ and $\F:=\F_{\emptyset}$ are called the minimal parabolic subalgebra and the maximal flag manifold, respectively.

An element of $\mathfrak{g}$ of the form $Y=\Ad(g)H$ with $g\in G$ and 
$H\in \cl\mathfrak{a}^{+}$ is called a {\bf split element}. When $%
H\in \mathfrak{a}^{+}$ we say that $Y$ is a {\bf split regular} element. The flow induced by a split element $H\in \cl\mathfrak{a}^{+}$ on $%
\mathbb{F}_{\Theta}$, is given by 
\begin{equation*}
(t,\Ad(g)\mathfrak{p})\mapsto \Ad(\mathrm{e}^{tH}g)\mathfrak{p}_{\Theta}.
\end{equation*}%
It turns out that the associated vector field is a gradient vector field
with respect to an appropriate Riemannian metric on $\mathbb{F}_{\Theta}$.

The connected components of the fixed point set of this flow are given by%
\begin{equation*}
\fix_{\Theta}(H,w)=Z_{H}\cdot wb_{\Theta}=K_{H}\cdot wb_{\Theta},\quad w\in \mathcal{W}.
\end{equation*}%
Here $b_{\Theta}$ is the origin of $\F_{\Theta}$, $Z_{H}$ is the centralizer of $H$ in $G$ and $K_{H}=Z_{H}\cap K$. The sets $\fix_{\Theta}(H, w)$ are in bijection with the double coset space $\mathcal{W}_{H}\backslash \mathcal{W}/\WC_{\Theta}$ where $\mathcal{W}_{H}$ and $\WC_{\Theta}$ are the subgroups of the
Weyl group generated, respectively, by the simple roots in $\Theta(H):=\{\alpha \in \Lambda, \;\alpha (H)=0\}$ and in $\Theta$.

Each component $\fix_{\Theta}(H, w)$ is a compact connected submanifold of $\mathbb{F}_{\Theta}$. Moreover, for $Y=\Ad(g)H$ with $H\in \cl\mathfrak{a}%
^{+}$ we get that 
\begin{equation*}
\fix_{\Theta}(Y, w)=\fix_{\Theta}(\Ad(g)H, w)=g\cdot \fix_{\Theta}(H, w).
\end{equation*}%
When $H$ is a split regular element the set $\fix_{\Theta}(H, w)$ reduces to the
point $wb_{\Theta}$ and consequently $\fix_{\Theta}(Y, w)=gwb_{\Theta}$.

\subsubsection{Semigroups}

Here we introduce the notion of control and chain control sets given by a semigroup. Let $\SC$ be a semigroup of a Lie group $G$ and $M$ a space provided with a $G$-transitive action. A $\SC$-control set on $M$ is a
subset $D\subset M$ satisfying

\begin{enumerate}
\item[$(i)$] $\inner D\neq \emptyset $

\item[$(ii)$] $D\subset \cl(\SC x)$ for all $x\in D$

\item[$(iii)$] $D$ is maximal (w.r.t.~set inclusion) with the properties $(i)$
and $(ii)$.
\end{enumerate}

A $\SC$-control set $D$ is said to be {\bf effective} if the set $D_{0}=\{x\in D,\;
\;x\in (\inner \SC)x\}$ is nonempty. The subset $D_{0}$ is said to be the 
\textbf{core} of $D$.

For semigroups with nonempty interior of a semisimple Lie group $G$ we have the following result from \cite{smt}. 

\begin{theorem}
\label{controlsets}
For any $w\in \mathcal{W}$ there is an effective $\SC$-control set $D_{\Theta}(w)\subset\F_{\Theta}$ whose core is given by
$$D_{\Theta}(w)_{0}=\left\{ \fix_{\Theta}(Y, w);\; \; \;Y\mbox{ is split regular and }\mathrm{%
e}^{Y}\in \inner \SC\right\}.$$
Moreover, $D_{\Theta}(1)$ is the only invariant $\SC$-control set, $D_{\Theta}(w_0)$ the only invariant $\SC^{-1}$-control set and any effective $\SC$-control set in $\F_{\Theta}$ is of the form $D_{\Theta}(w)$ for some $w\in\WC$.
\end{theorem}

It follows from the above that there exists a unique $\Theta(\SC)\subset\Lambda$ such that the set $\{w\in \mathcal{W}: D(w)=D(1)\}\subset\WC$ coincides with the subgroup $\mathcal{W}_{\Theta(\SC)}$ of $\WC$. Moreover,  
\begin{equation*}
\mathcal{W}_{\Theta(\SC)}w_{1}\WC_{\Theta}=\mathcal{W}_{\Theta(\SC)}w_{2}\WC_{\Theta}\Leftrightarrow D_{\Theta}(w_{1})=D_{\Theta}(w_{2})
\end{equation*}
and the number of effective control sets in $\mathbb{F}_{\Theta}$ is equal to the number of elements in the coset space $\mathcal{W}_{\Theta(S)}\setminus \mathcal{W}/\WC_{\Theta}$. The subset $\Theta(S)$ is called the {\bf flag type of the semigroup $\SC$}.

We also have a concept of chain control sets associated with a semigroup as follows: Let $\mathcal{F}$ be a family of subsets of $\SC$. For $x,y\in M$, $\varepsilon >0$ and $A\in \mathcal{F}$ a $(\SC,\varepsilon ,A)$\textbf{-chain}
from $x$ to $y$ is given by an integer $n\in \mathbb{N}$, $x_{0},\ldots, x_{n}\in M$, $g_{0},\ldots ,g_{n-1}\in A$ such that $x_{0}=x$, $x_{n}=y$
and 
$$\varrho (g_{i}x_{i},x_{i+1})<\varepsilon,  \;\;\mbox{ for }\;\; i=0,\ldots ,n-1.$$ 
A subset $E_{\FC}\subset M$ is called a $\mathcal{F}$-\textbf{chain control set} of $\SC$ if it satisfies

\begin{enumerate}
\item[$(i)$] $\inner E_{\FC}\neq \emptyset $.

\item[$(ii)$] For all $x,y\in E_{\FC}$, $\varepsilon >0$ and $A\in \mathcal{F}$
there exists an $(\SC,\varepsilon ,A)$-chain from $x$ to $y$.

\item[$(iii)$] $E_{\FC}$ is maximal (w.r.t.~set inclusion) with the properties $(i)$
and $(ii)$.
\end{enumerate}

A $\mathcal{F}$-chain control set is said to be \textbf{effective} if it contains an effective $\SC$-control set.

When $G$ is a noncompact semisimple Lie group with finite center, we have the following refined result (see \cite{bsm}, Proposition 4.7).

\begin{proposition}
\label{chaincontrolsets}
For each $w\in \mathcal{W}$ there exists an effective $\mathcal{F}$%
-chain control set $E_{\FC, \Theta}(w)$ for $\SC$ on $\mathbb{F}_{\Theta}$ such that $D_{\Theta}(w)\subset E_{\FC, \Theta}(w)$.

Moreover, the set $\mathcal{W}_{\mathcal{F}}(\SC)=\{w\in \mathcal{W}:E_{\FC}(w)=E_{\FC}(1)\}$ is a
subgroup of $\mathcal{W}$ and 
\begin{equation*}
\mathcal{W}_{\mathcal{F}}(\SC)w_{1}\WC_{\Theta}=\mathcal{W}_{\mathcal{F}%
}(\SC)w_{2}\WC_{\Theta}\Leftrightarrow E_{\FC, \Theta}(w_{1})=E_{\FC, \Theta}(w_{2}).
\end{equation*}
\end{proposition}

The unique subset $\Theta(\FC)\subset\Lambda$ such that $\WC_{\Theta(\FC)}=\mathcal{W}_{\mathcal{F}}(S)$ is called the $\mathcal{F}$-{\bf flag
type of the semigroup $\SC$}. The number of chain control sets in $\mathbb{F}_{\Theta}$
is equal to the number of elements in the coset space $\mathcal{W}_{\mathcal{F}}(S)\setminus \mathcal{W}/\WC_{\Theta}$. When there is no change of confusion, we denote only by $E_{\Theta}(w)$ the above $\FC$-chain control sets on $\F_{\Theta}$.

By Proposition \ref{chaincontrolsets} above, we have that $\Theta(\SC)\subset\Theta(\FC)$ for any family of subsets $\FC$ of $\SC$. 

\begin{remark}
The above definitions of $\SC$-control sets and $\mathcal{F}$-chain control sets differs from their analogous ones for control systems, mainly by the assumption that both, $\SC$-control sets and $\mathcal{F}$-chain control sets, have nonempty interior which is not necessarily true for control systems.
\end{remark}

\subsubsection{Flow on flag bundles}

Let $X$ be a compact metric space and $\phi :\mathbb{R}\times X\rightarrow
X,\;(t,x)\mapsto \phi _{t}(x)$ be a continuous flow. A compact subset $%
C\subset X$ is called isolated invariant if $\phi _{t}(C)\subset C$ for all $%
t\in \mathbb{R}$ and if there is a neighborhood $V$ of $C$ that satisfy 
\begin{equation*}
\phi _{t}(x)\in V\; \; \mbox{ for all }\; \;t\in \mathbb{R}\; \; \Rightarrow
\; \;x\in C.
\end{equation*}%
A {\bf Morse decomposition} of $\phi $ is a finite collection $\{ \mathcal{M}%
_{1},\ldots ,\mathcal{M}_{n}\}$ of nonempty pairwise disjoint isolated
invariant compact sets satisfying

\begin{itemize}
\item[(A)] for all $x\in X$ the $\omega^*$ and $\omega$-limit sets $%
\omega^*(x)$ and $\omega(x)$, respectively, are contained in $\bigcup_{i=1}^n%
\mathcal{M}_i$.

\item[(B)] if there are $\mathcal{M}_{j_0}, \ldots, \mathcal{M}_{j_l}$ and $%
x_1, \ldots, x_l\in X\backslash \bigcup_{i=1}^n\mathcal{M}_i$ with 
\begin{equation*}
\omega^*(x_i)\subset \mathcal{M}_{j_{i-1}}\; \; \; \mbox{ and }\; \; \;
\omega(x_i)\subset \mathcal{M}_{j_i}
\end{equation*}
for $i=1, \ldots, l$ then $\mathcal{M}_{j_0}\neq \mathcal{M}_{j_l}$.
\end{itemize}

The elements in a Morse decomposition are called {\bf Morse sets}.

Let $G$ be a connected semisimple noncompact Lie group with finite center and $\pi :Q\rightarrow X$ a $G$-principal bundle, where $X$ is a compact
metric space. The Lie group $G$ acts continuously on from the right on $Q$, this action preserves the fibers, and is free and transitive on each fiber.
The \textbf{flag bundle} $\mathbb{E}_{\Theta}=Q\times _{G}\mathbb{F}_{\Theta}$ with typical fiber $\mathbb{F}_{\Theta}$ is given by $(Q\times \mathbb{F}_{\Theta})/\! \sim $, where $(q_{1},b_{1})\sim (q_{2},b_{2})$ iff there exists $g\in G$ with $q_{1}=q_{2}\cdot g$ and $b_{1}=g^{-1}\cdot b_{2}$.

Now let $\phi_t:Q\rightarrow Q$ be a flow of automorphisms, i.e., $\phi_t(q\cdot g)=\phi_t(q)\cdot g$, and assume that the induced flow on $X$
is chain transitive. For the induce flow $\phi^{\Theta}_t:\mathbb{E}_{\Theta}\rightarrow%
\mathbb{E}_{\Theta}$ we have the following result (see Thm. 9.11 of \cite{cbsm}, Thm. 5.2 of \cite{smls}).

\begin{theorem}
\label{morsesets} There exist $H_{\phi}\in \cl(\mathfrak{a}^+)$ and a
continuous $\phi$-invariant map 
\begin{equation*}
h:Q\rightarrow \Ad(G)H_{\phi}, \; \; \;h(\phi_t(q))=h(q),
\end{equation*}
satisfying $h(q\cdot g)=\Ad(g^{-1})h(q), \;q\in Q, \;g\in G$ and such that the induced flow on $\mathbb{E}_{\Theta}$ admits a finest Morse decomposition whose elements are given fiberwise by 
\begin{equation*}
\mathcal{M}_{\Theta}(w)_{\pi(q)}=q\cdot \fix_{\Theta}(h(q), w).
\end{equation*}
\end{theorem}

The subset $\Theta(\phi)=\Theta(H_{\phi })$ is called the {\bf flag type of the flow $\phi$}. The number of the Morse sets in $\mathbb{E}_{\Theta}$ is equal to the number of elements in the coset space $\mathcal{W}_{\Theta(\phi)}\backslash \mathcal{W}/\WC_{\Theta}$.

\section{The main result}

Here we show that for induced control systems on flag manifolds, both concepts of control and chain control sets above coincide implying that one could analize the behaviour of invariant control systems via semigroup theory. In particular, the main theorem implies that any chain control sets of a locally accessible induce system contains a control set with nonempty interior, which is in general not necessarely true (see Example 3.4.2 of \cite{ck}). Moreover, when $\Theta(\SC)=\Theta(\phi)$ every chain control set is the closure of the only control set with nonempty interior that it contains.

Let $G$ be a noncompact connected semisimple Lie group with finite center and consider the right-invariant system $\Sigma$ on $G$. Moreover, assume that $\Sigma$ is a locally accessible right-invariant system on $G$, and therefore, that all the induced system $\Sigma_{\Theta}$ are also locally accessible. From here on $\SC$ will stand for the semigroup given by $\OC^+(1)$, which by the local accessibility condition has nonempty interior in $G$.

\begin{proposition}
\label{control}
A subset $D\subset\F_{\Theta}$ is an effective $\SC$-control set if and only if it is a control set of $\Sigma_{\Theta}$ with nonempty interior.
\end{proposition}

\begin{proof}
By Proposition 3.2.4 of \cite{ck} and the invariance of $\Sigma$ we have that any effective $\SC$-control set in $\F_{\Theta}$ is in fact a control set of $\Sigma_{\Theta}$ with nonempty interior. 

Reciprocally, if $D\subset\F_{\Theta}$ is a control set with nonempty interior of $\Sigma_{\Theta}$, we have by invariance that $D$ is a $\SC$-control set and we only have to show that it is effective. 

By the local accessibility assumption, we have that $\inner D$ is dense in $D$ (see Lemma 3.2.13 of \cite{ck}). Moreover, since $\inner\SC$ is dense in $\SC$ and $1\in \SC$ we have that 
$$\left(\inner\SC\right)^{-1}\cdot x\cap D\neq\emptyset, \;\;\mbox{ for any }\;\;x\in \inner D.$$
Therefore, by property (ii) in the definition of control sets, $\left(\inner\SC\right)x\cap\left(\inner\SC\right)x\neq\emptyset$ for any $x\in\inner D$ which implies that $D$ is in fact effective concluding the proof.   
\end{proof}

\begin{remark}
We should note that the above result does not use the fact that the Lie group $G$ is semisimple, which implies that it remains true for an induced system on an arbitrary homogeneous space. 
\end{remark}

Now we turns out to the problem of the chain control sets of $\Sigma_{\Theta}$. Consider the $G$-principal bundle $\pi:\mathcal{U}\times G\rightarrow \mathcal{U}$, where $\pi$ is the projection on the first factor and $G$ acts from the right on $\mathcal{U}\times G$, trivially, as $(u, g)\cdot h=(u, gh) $. The invariance of $\Sigma$ implies that $\varphi(t, gh, u)=\varphi(t, g, u)h$ which implies that the control flow $\phi_t$ is a flow of automorphisms on $\UC\times G$. Moreover, the induced flow on $\mathcal{U}$ is the shift flow, that is chain transitive. Therefore, Theorem \ref{morsesets} can be applied in order to obtain the Morse sets of the induced control flows $\phi^{\Theta}$ for any $\Theta\subset\Lambda$.

Let $\mathcal{F}$ be the family of subset of $\SC$ given by 
\begin{equation*}
\mathcal{F}:=\left \{ \bigcup_{t>\tau }\mathcal{O}%
_{t}^{+}(e),\; \; \tau > 0\right \} .
\end{equation*}

Now we prove our main result.

\begin{theorem}
\label{chaincontrol} A subset $E\subset\F_{\Theta}$ is a $\FC$-chain control set if and only if it is a chain control set of $\Sigma_{\Theta}$.
\end{theorem}

\begin{proof} For a given subset $E$ of $\F_{\Theta}$ the lift of $E$ is given by 
\begin{equation*}
\mathcal{E}=\{(u, x)\in \mathcal{U}\times \mathbb{F}_{\Theta}; \; \; \; \varphi^{\Theta}(\mathbb{R}, x, u)\subset E\}.
\end{equation*}
By Theorem 4.3.11 of \cite{ck} there exists a bijection between the Morse sets of $\phi^{\Theta}$ and the chain control sets of the system $\Sigma_{\Theta}$ given by:

If $\mathcal{M}_{\Theta}$ is a Morse set of $\phi^{\Theta}$, then $\pi_2(\mathcal{M}_{\Theta})$ is a chain control set of $\Sigma_{\Theta}$, where $\pi_2:\mathcal{U}\times \mathbb{F}_{\Theta}\rightarrow \mathbb{F}_{\Theta}$ is the projection on the second factor. Reciprocally, if $E$ is a chain control set of $\Sigma_{\Theta}$ its lift $\EC$ is a Morse set of $\phi^{\Theta}$.

We divide the rest of our proof in three steps:

{\bf Step 1:} If $E_{\FC}$ is a $\FC$-chain control set in $\F_{\Theta}$ then its lift $\EC_{\FC}$ is contained in a Morse set of $\phi^{\Theta}$.

By condition (ii) in the definition of $\FC$-chain control sets and the right-invariance of $\Sigma$ we have that $E_{\FC}$ is controlled chain transitive. By Theorem 4.3.11 one gets, in particular, that its the lift $\EC_{\FC}$ is chain transitive for the flow $\phi^{\Theta}$. Moreover, since $\EC_{\FC}$ is certainly $\phi^{\Theta}$-invariant, Theorem B.2.26 of \cite{ck} implies that $\EC_{\FC}$ is contained in some Morse set, since the Morse sets are the maximal $\phi^{\Theta}$-invariant chain transitive subsets of $\UC\times\F_{\Theta}$.

{\bf Step 2:} Any effective $\FC$-chain control set in $\F_{\Theta}$ is a chain control set of $\Sigma_{\Theta}$.

Let $E$ be an effective $\mathcal{F}$-chain control in $\F_{\Theta}$ and $D\subset E$ an effective control set. By Proposition \ref{control}, $D$ is a control set of $\Sigma_{\Theta}$ with nonempty interior. By Corollary 4.3.12 of \cite{ck} there exists a chain control set $E'$ of $\Sigma_{\Theta}$ such that $D\subset E'$. By the very definition of $\mathcal{F}$-chain control sets and the fact that $\inner D\neq \emptyset$ we get that $E'$ is an effective $\mathcal{F}$-chain control set in $\F_{\Theta}$ which by maximality implies that $E'\subset E$. By the previous item the inclusion $E\subset E'$ must always happens and so $E=E'$ is a chain control set of $\Sigma_{\Theta}$.

{\bf Step 3:} Any  chain control set of $\Sigma_{\Theta}$ is an effective $\FC$-chain control set.

Let $E$ be a chain control set of $\Sigma_{\Theta}$. In order to show that $E_{\Theta}$ is a $\FC$-chain control set, it is enough to show that $E$ contains an effective $\SC$-control set.

Define $\mathsf{h}:\mathcal{U}\rightarrow \Ad(G)H_{\phi }$ by $\mathsf{h}(u):=h(u, 1)$, where $h$ is the map given by Theorem \ref{morsesets}
applied to the control flow $\phi$ on the principal bundle $\UC\times G$. The map $\mathsf{h}$ is continuous and we have
that $\Ad(\varphi (t,e,u))\mathsf{h}(u)=\mathsf{h}(\theta_{t}u)$.
Moreover, the relation between the Morse sets and the chain control sets together with Theorem \ref{morsesets} give us that
\begin{equation*}
E=\bigcup_{u\in\UC} \fix_{\Theta}(\mathsf{h}(u), w), \;\;\mbox{ for some }\;\;w\in\WC.
\end{equation*}%

Let $g\in \inner \SC$ and consider $u\in\mathcal{U}$ and $\tau >0$ such that $g=\varphi(\tau, e, u)$. By extending periodically $u$ to a $\tau$-periodic control function $u^*\in\UC$ we have that 
$$g^n=\varphi(\tau, 1, u)^n=\varphi(n\tau, 1, u^*)\in\inner\SC.$$
Therefore, 
\begin{equation*}
g\fix_{\Theta}(\mathsf{h}(u^*),w)=\fix_{\Theta}(\Ad(g)\mathsf{h}(u^*), w)=\fix_{\Theta}(\mathsf{h}(\theta _{\tau }u^*), w)=\fix_{\Theta}(\mathsf{h}(u^*), w),
\end{equation*}%
and so $\fix_{\Theta}(\mathsf{h}(u^*), w)$ is a $g$-invariant compact set. Hence, there exists a nonempty subset $\Omega \subset \fix_{\Theta}(\mathsf{h}(u^*), w)$, that is
minimal for the $g$-action in $\mathbb{F}_{\Theta}$. By Proposition 2.3 of \cite{smt}, the set $\Omega$ has to be contained in the interior of a $\SC$-control set $D$ in $\F_{\Theta}$. Since $\Omega \subset E$ we must have that $D\subset E$ showing that $E$ is in fact an effective $\FC$-chain control set which concludes the 	proof.
\end{proof}

As a direct consequence we have the following.

\begin{corollary}
Any chain control set of system $\Sigma_{\Theta}$ contains a control set with nonempty interior.
\end{corollary}

We have also that the flag type of $\phi$ and the $\FC$-flag type of $\SC$ coincides.

\begin{corollary}
\label{flagtype} With the previous assumptions, it holds that 
\begin{equation*}
\Theta(\mathcal{F})=\Theta(\phi).
\end{equation*}
\end{corollary}

\begin{proof}
Since for any two given subsets $\Theta_1, \Theta_2\subset\Lambda$ we have that $\Theta_1=\Theta_2$ if and only if $\WC_{\Theta_1}=\WC_{\Theta_2}$ it is enough for us to show that $\WC_{\FC}(\SC)=\WC_{\Theta(\phi)}$.

By Theorem \ref{chaincontrol} we have that the set of the chain control sets of the induced systems in $\F$ and the set of the $\FC$-chain control sets in $\F$ coincides. Consequently
\begin{equation*}
\left|\mathcal{W}_{\mathcal{F}}(\SC)\backslash \mathcal{W}\right|=\left|%
\mathcal{W}_{\Theta(\phi)}\backslash \mathcal{W}\right|\mbox{ implying that }%
\left|\mathcal{W}_{\mathcal{F}}(\SC)\right|=\left|\mathcal{W}_{\Theta(\phi)}\right|.
\end{equation*}
Then it is enough to show that $\mathcal{W}_{\mathcal{F}}(\SC)\subset \mathcal{%
W}_{\Theta(\phi)}$.

Since $\pi _{2}(\mathcal{M}(1))$ is $\SC$-invariant, it contains the only $\SC$-invariant control set in $\F$, and so $\pi_{2}(\mathcal{M}(1))=E(1)$.
Let then $w\in \mathcal{W}_{\mathcal{F}}(\SC)$. By definition $E(w)$ is the only $\mathcal{F}$-chain control set containing $D(w)$. Moreover, since $w\in \mathcal{W}_{\mathcal{F}}(\SC)$ we obtain $E(w)=E(1)=\pi_{2}(\mathcal{M}(1))$. 

Consequently 
\begin{equation*}
D(w)\subset \,\bigcup_{u\in \mathcal{U}}\fix(\mathsf{h}(u), 1).
\end{equation*}

Since any other choice of positive Weyl chamber just conjugates the flag
types, we can assume w.l.o.g. that 
\begin{equation*}
\inner \SC\cap \exp \mathfrak{a}^{+}\neq \emptyset .
\end{equation*}
Then $wb_{0}\in D(w)_{0}$ and there exists $u\in \mathcal{U}$ such that $wb_{0}\in \fix(\mathsf{h}(u),1)$, which gives us that $wb_{0}=kb_{0}$ where $%
k\in K$ is such that $\Ad(k)H_{\phi}=\mathsf{h}(u)$. The equality $wb_{0}=kb_{0}$ implies that $wH_{\phi}=\Ad(k)H_{\phi}=\mathsf{h}(u)$
and consequently  
\begin{equation*}
\fix(\mathsf{h}(u),1)=\fix(wH_{\phi},1)=(wZ_{H_{\phi}}w^{-1})wb_{0}=\fix(wH_{\phi },w)=\fix(\mathsf{h}(u), w).
\end{equation*}%
Hence $w\in \mathcal{W}_{\Theta(\phi)}$ implying that $\mathcal{W}_{\mathcal{F}}(S)\subset \mathcal{W}_{\Theta(\phi)}$ and concluding the proof.
\end{proof}

The above corollary implies in particular that $\Theta(\SC)\subset\Theta(\phi)$. We are now interested to see what happens when the equality holds, that is, when any chain control set of $\Sigma_{\Theta}$ contains exactly one control set with nonempty interior. 

\begin{definition}
Let $\Theta\subset\Lambda$ and $w\in\WC$. We say that the chain control set $E_{\Theta}(w)$ is {\bf (uniformly) hyperbolic} if for each $(u,x)\in\MC_{\Theta}(w)$ there exists a decomposition%
\begin{equation*}
  T_x\F_{\Theta} = E^-_{u,x} \oplus E^+_{u,x}%
\end{equation*}
such that the following properties hold:%
\begin{enumerate}
\item[(H1)] $(\rmd\varphi_{t,u})_x E^{\pm}_{u,x} = E^{\pm}_{\phi_t(u,x)}$ for all $t\in\R$ and $(u,x)\in\MC_{\Theta}(w)$.%
\item[(H2)] There exist constants $c\in(0,1]$, $\lambda>0$ such that for all $(u,x)\in\MC_{\Theta}(w)$ we have%
\begin{equation*}
  \left|(\rmd\varphi_{t,u})_x v\right| \leq c^{-1}\rme^{-\lambda t}|v| \mbox{\quad for all\ } t\geq0,\ v\in E^-_{u,x},%
\end{equation*}
and%
\begin{equation*}
  \left|(\rmd\varphi_{t,u})_x v\right| \geq c\rme^{\lambda t}|v| \mbox{\quad for all\ } t\geq0,\ v\in E^+_{u,x}.%
\end{equation*}
\item[(H3)] The linear subspaces $E^{\pm}_{u,x}$ depend continuously on $(u,x)$, i.e., the projections $\pi^{\pm}_{u,x}:T_x\F_{\Theta} \rightarrow E^{\pm}_{u,x}$ along $E^{\mp}_{u,x}$ depend continuously on $(u,x)$.%
\end{enumerate}
\end{definition}

In \cite{ask} it is shown that any chain control set $E_{\Theta}(w)$ is parcially hyperbolic, that is, the above decomposition of the tangent spaces have a third subspace that corresponds to the tangent space of the submanifold $\fix_{\Theta}(h(u), w)$. By Theorem 5.6 of \cite{ask}, hyperbolicity happens if and only if $\langle\Theta(\phi)\rangle\subset w\langle\Theta\rangle$ if and only if for any $u\in\UC$ the set $\fix_{\Theta}(h(u), w)$ consists of only one point. Consequently, if $E_{\Theta}(w)$ is hyperbolic, Theorem 6.1 of \cite{ask} implies that $\cl(D_{\Theta}(w))=E_{\Theta}(w)$.

The above together with the assumption that the flag types agree give us the following result concerning invariant control sets.

\begin{proposition}
\label{invariant}
If $\Theta(\SC)=\Theta(\phi)$ then 
$$D_{\Theta}(1)=E_{\Theta}(1)\;\;\;\mbox{ and }\;\;\;\cl\left(D_{\Theta}(w_0)\right)=E_{\Theta}(w_0)$$
for any $\Theta\subset\Lambda$.
\end{proposition}

\begin{proof}
We will show only the first equality since the other is obtained from it by considering the negative time systems. By the above discussion we have that and Theorem \ref{chaincontrol} we have that the chain control set $E_{\Theta(\phi)}(1)\subset \F_{\Theta(\phi)}$ is uniformly hyperbolic and so it is the closure of a control set. Therefore, 
$$E_{\Theta(\phi)}(1)=\cl\left(D_{\Theta(\phi)}(1)\right)=D_{\Theta(\phi)}(1),$$
since $D_{\Theta(\phi)}(1)$ is closed. Moreover, Theorem 4.3 of \cite{smt} assures that $\pi_{\Theta(\SC)}^{-1}\left(D_{\Theta(\SC)}(1)\right)=D(1)$ and Proposition 8.12 of \cite{cbsm}, together with the relation between the chain control sets and the Morse sets of the control flow, implies  that $\pi_{\Theta(\phi)}^{-1}\left(E_{\Theta(\phi)}(1)\right)=E(1)$, where $\pi_{\Theta(\SC)}:\F\rightarrow\F_{\Theta(\SC)}$ and $\pi_{\Theta(\phi)}:\F\rightarrow\F_{\Theta(\phi)}$ are the canonical projections. Therefore, if $\Theta(\SC)=\Theta(\phi)$ we have that 
$$E(1)=\pi_{\Theta(\phi)}^{-1}\left(E_{\Theta(\phi)}(1)\right)=\pi_{\Theta(\SC)}^{-1}\left(D_{\Theta(\phi)}(1)\right)=D(1).$$
Since for any $\Theta\subset\Lambda$ we have that $D_{\Theta}(1)=\pi_{\Theta}(D(1))$ and $E_{\Theta}(1)=\pi_{\Theta}(E(1))$ the result follows.
\end{proof}

\bigskip
For any $w\in\WC$, the {\bf domain of attraction} of $D(w)$ is the subset of $\F$ given by
$$\AC(D(w))=\left\{x\in\F, \;gx\in D(w) \;\;\mbox{ for some }\;g\in\SC\right\}.$$
The following lemma relates the domain of attraction and the core of an effective control set.

\begin{lemma}
\label{intersection}
For any $w\in\WC$ it holds that 
$$D(w)_0=\AC(D(w))\cap \AC^*(D^*(w)), $$
where $D^*(w)$ is the unique $\SC^{-1}$-control set in $\F$ whose core is $D^*(w)_0=D(w)_0$ and $\AC^*(D^*(w))$ its domain of attraction. 
\end{lemma}

\begin{proof}
Since $D(w)_0=D^*(w)_0$ we already have that 
$$D(w)_0\subset\AC(D(w))\cap \AC^*(D^*(w)).$$
For any $x\in\AC(D(w))$ there exists $g\in\SC$ such that $gx\in D(w)$. By using condition (ii) in the definition of $\SC$-control sets implieswe get that $D(w)_0\subset D(w)\subset\cl(\SC gx)$. Since $D(w)_0$ is an open set and $\inner\SC$ is dense in $\SC$, we have that $D(w)_0\cap \inner\SC gx\neq\emptyset$. 

Analogously, if $x\in \AC^*(D^*(w))$ we get that $D^*(w)_0\cap\inner\SC^{-1}(g'x)\neq\emptyset$ for some $g'\in\SC^{-1}$. Therefore, if $x\in\AC(D(w))\cap\AC^*(D^*(w))$ there exists $h'\in\SC^{-1}$ and $h\in\SC$ such that $h'g'x, hgx\in D(w)_0$, since $D(w)_0=D^*(w)_0$. Being that in $D(w)_0$ we have controllability, there exists $k\in\SC$ such that 
$$khgx=h'g'x\Rightarrow ((h'g')^{-1}khg)x=x.$$
However, since $\inner\SC$ in $\SC$-invariant we get that $(h'g')^{-1}khg\in\inner\SC$ showing that $x\in D(w)_0$ and concluding the proof.
\end{proof}

\bigskip
By Theorem 6.3 of \cite{sm}, the above domain of attraction can be characterized as follows: For a finite sequence $\alpha_1, \ldots, \alpha_n$ in $\Lambda$ let us denote by $s_1, \ldots, s_n$ the reflections with respect these roots and consider $P_i:=P_{\{\alpha_i\}}$. The corresponding flag manifold is $\F_i=G/P_i$ and we denote by $\pi_i$ the canonical projection of $\F$ onto $\F_i$. Moreover, for a given subset $X$ of $\F$ we denote by $\gamma_i(X):=\pi_i^{-1}\pi_i(X)$ the operation of exhausting the subset $X$ with the fibers of $\pi_i$.

Now, take $w\in\WC$ and consider the reduced expressions $w=r_m\cdots r_1$ and $w_0w=s_n\cdots s_1$ with exhausting maps $\gamma_i'$ and $\gamma_j$, respectively. It holds that 
\begin{equation}
\label{domain}
\AC(D(w))=\gamma_1\cdots\gamma_n(D(w_0))\;\;\mbox{ and }\;\;\AC^*(D^*(w))=\gamma'_1\cdots\gamma'_m(D(1))
\end{equation}

Concerning chain control sets, from Propositions 9.9 and 9.10 of \cite{cbsm} and the fact that $\pi_2(\MC(w))=E(w)$ it is straightforward to see that 
\begin{equation}
\label{attraction}
E(w)\subset\gamma_1'\cdots\gamma_m'\left(E(1)\right)\cap \gamma_1\cdots\gamma_n\left(E(w_0)\right).
\end{equation}

We can now prove the following.

\begin{theorem}
It holds that $\Theta(\phi)=\Theta(\SC)$ if and only if  
\begin{equation}
\label{closure}
\cl\left(D_{\Theta}(w)\right)=E_{\Theta}(w)
\end{equation}
for any $\Theta\subset\Lambda$ and any $w\in\WC$.
\end{theorem}

\begin{proof}
If the equality (\ref{closure}) is true for any $\Theta\subset\Lambda$ we have in particular that $|\WC_{\Theta(\SC)}|=|\WC_{\Theta(\phi)}|$. Since $\Theta(\SC)\subset\Theta(\phi)$ always holds, we must have $\WC_{\Theta(\SC)}=\WC_{\Theta(\phi)}$ and conquesently $\Theta(\phi)=\Theta(\SC)$.

If reciprocally we assume that $\Theta(\phi)=\Theta(\SC)$, Proposition \ref{invariant} implies that the closure of any invariant control is a chain control set which by Proposition \ref{invariant} and equations (\ref{domain}) and (\ref{attraction}) imply that
$$E(w)\subset\AC^*(D^*(w))\cap \gamma_1\cdots\gamma_n\left(\cl(D(w_0))\right).$$
Moreover, since $\pi_i:\F\rightarrow\F_i$ is a continous open map between compact topological spaces, we have that $\gamma_i(\cl(X)))=\cl(\gamma_i(X))$ for any subset $X$ of $\F$. Therefore 
$$\gamma_1\cdots\gamma_n\left(\cl(D(w_0))\right)=\cl\left(\gamma_1\cdots\gamma_n(D(w_0))\right)$$
and so
$$E(w)\subset\AC^*(D^*(w))\cap \cl\left(\AC(D(w))\right).$$
Since by Lemma \ref{intersection} it holds that $\AC^*(D^*(w))\cap \AC(D(w))=D(w)_0$ we get that 
$$E(w)\subset\AC^*(D^*(w))\cap \cl\left(\AC(D(w))\right)\subset \cl(D(w)_0)=\cl(D(w))\subset E(w)$$
showing that $\cl(D(w))=E(w)$ for any $w\in\WC$.

Since for any $\Theta\subset\Lambda$ and any $w\in\WC$ we have that
$$\pi_{\Theta}(D(w))=D_{\Theta}(w)\;\;\;\mbox{ and }\;\;\;\pi_{\Theta}(E(w))=E_{\Theta}(w)$$
the result follows.
\end{proof}

\begin{remark}
Let $A, B$ are $n\times n$-matrices with trace zero and consider on $\R^n$ the bilinear system
$$\dot{x}(t)=\left(A+u(t)B\right)x(t).$$
It is well known that the behavious of such system is intrinsically connected with its induced control-affine system on the Grassmannians spaces. Since $A, B\in sl(n)$, we have a well defined right-invariant system on the semisimple Lie group $Sl(n)$. It is well known that any Grassmanian space is a flag manifold $\F_{\Theta}$ of $Sl(n)$ for some appropriated $\Theta$. Moreover, any system induced by the bilinear one on a Grassmanian space coincides with the system induced by the right-invariant system on the corresponding flag. Therefore, the above results can be used, for instance, in order to analyze the behaviour of bilinear systems on $\R^n$.
\end{remark}


\begin{thebibliography}{99}


\bibitem{bsm} \textsc{Braga Barros, C.J. and L.A.B. San Martin}, \emph{Chain
control sets for semigroup actions}, Mat. Apl. Comp., 15 (1996),
pp.~257-276.

\bibitem{cbsm} \textsc{Braga Barros, C.J. and L.A.B. San Martin}, \emph{%
Chain transitive sets for flows on flag bundles}, Forum Math., 19 (2007),
pp.~19-60.

\bibitem{bro} \textsc{Brocke, R. W.}, \emph{System theory on group manifolds and coset spaces}, SIAM Journal on Control, {\bf 10} No 2 (1972), pp. 265 - 284.


\bibitem{ck} \textsc{Colonius, F. and W. Kliemann}, \emph{The Dynamics of
Control}, Birkhauser, (2000).

\bibitem{ask} \textsc{Da Silva, A. and C. Kawan}, \emph{Hyperbolic chain
control sets on flag manifolds}, Journal of Dynamics and Control Systems, {\bf 21} No. 4 (2015), pp. 1-21.

\bibitem{ask1} \textsc{Da Silva, A. and C. Kawan}, \emph{Invariance entropy of hyperbolic control sets}, Discrete and Continuous Dynamical Systems {\bf 36} No.1 (2016), pp. 97-136.




\bibitem{Jurd} \textsc{Jurdjevic V. and H. Sussmann}, \emph{Control systems on Lie groups},
Journal of Differential Equations, {\bf 12} (1972), pp 313-329.

\bibitem{Ka2} \textsc{Kawan, C.}, \emph{Invariance Entropy for Deterministic
Control Systems -- An Introduction}, Lecture Notes in Mathematics, 2089,
Springer-Verlag, Berlin, 2013.

\bibitem{psm} \textsc{Patrao, M. and L.A.B. San Martin}, \emph{Semiflows on
Topological Spaces:\ Chain Transitivity and Semigroups}, Journal of Dynamics
and Differential Equations, 19 (1), (2007).

\bibitem{sm} \textsc{San Martin, L.}, \emph{Order and domains of attractions of control sets in flag manifolds}, Journal of Lie theory {\bf 8} No. 2, (1998) pp. 335-350.


\bibitem{smls} \textsc{San Martin, L.A.B. and L. Seco}, \emph{Morse and
Lyapunov spectra and dynamics on flag bundles}, Ergod. Th. \& Dynam. Sys. 30
(2010), 3, 893--922.

\bibitem{smt} \textsc{San Martin, L.A.B. and P.A. Tonelli}, \emph{Semigroup
actions on Homogeneous Spaces}, Semigroup Forum, 50 (1995), pp.~59--88.

\end{thebibliography}
\end{document}